\documentclass[12pt,a4paper]{article} 
\usepackage{amsmath,amsthm,amsfonts,amssymb,amscd}
\usepackage{bbm}
\usepackage{enumerate,enumitem}
\usepackage{url,cite}
\usepackage{fullpage,fancyhdr}
\usepackage[subnum]{cases}
\usepackage{color,float,soul}
\usepackage{graphicx}
\usepackage[margin=1in]{geometry}

\newcommand{\fenv}[1]%
{\ensuremath{\,\overrightarrow{\operatorname{env}}_{#1}}}
\newcommand{\benv}[1]%
{\ensuremath{\,\overleftarrow{\operatorname{env}}_{#1}}}
\theoremstyle{plain}
\newtheorem{theorem}{Theorem}[section]
\newtheorem{lemma}[theorem]{Lemma}
\newtheorem{corollary}[theorem]{Corollary}

\newtheorem{definition}{Definition}[section]

\newtheorem{example}{Example}[section]

\usepackage{cite}

\newcommand{\Gph}{{\rm\textbf{Gph}\,}}

\newcommand{\nexto}{\kern -0.54em}

\newcommand{\dZ}{{\cal Z \kern -0.7em Z}}
\newcommand{\dC}{{\rm\hbox{C \kern-0.8em\raise0.2ex\hbox{\vrule height5.4pt width0.7pt}}}}
\newcommand{\dQ}{{\rm\hbox{Q \kern-0.85em\raise0.25ex\hbox{\vrule height5.4pt width0.7pt}}}}

\usepackage{epsfig}
\usepackage{latexsym}

\newcommand{\RR}{\mathbb{R}}

\newcommand{\mat}[1]

\date{}
\begin{document}
\title{\Large\textrm Polytropic Dynamical Systems with Time Singularity}
\author{Oday Hazaimah 
\footnote{ Northern Illinois University, {\tt odayh982@yahoo.com}. https://orcid.org/0009-0000-8984-2500.} 
}
\maketitle
\begin{abstract}
In this paper we consider a class of second order singular homogeneous differential equations called the Lane-Emden-type with time singularity in the drift coefficient. Lane-Emden equations are singular initial value problems that model phenomena in astrophysics such as stellar structure and are governed by polytropics with applications in isothermal gas spheres. A hybrid method that combines two simple methods; Euler's method and shooting method, is proposed to approximate the solution of this type of dynamic equations. We adopt the shooting method to reduce the boundary value problem, then we apply Euler's algorithm to the resulted initial value problem to get approximations for the solution of the Lane-Emden equation. Finally, numerical examples and simulation are provided to show the validity and efficiency of the proposed technique, as well as the convergence and error estimation are analyzed. 
		
\end{abstract}
	
	
	\maketitle
	
\section{\textbf{Introduction}}
Laplace's equation and Poisson's equation are important examples of elliptic partial differential equations which used broadly in applied mathematics and theoretical physics, see, e.g., \cite{PeterOlver}. For instance, Poisson's equation used to calculate gravitational field in potential theory and can be seen as generalization of Laplace's equation. By removing or reducing dimensions from Poisson's equation, we obtain a second-order nonlinear differential equation called Lane-Emden-type equation (LE, for short). The Lane-Emden equation (a.k.a. polytropic dynamic equation) is one of the well studied classical dynamical systems that has many applications in nonlinear mathematical physics and non-Newtonian fluid mechanics (see, for instance, \cite{Ali,Asadpour,Bile,Datt,Davis,KouroshAmin,sand}). A preliminary study on the LE equations (polytropic and isothermal) was undertaken by astrophysicists Lane (1870) and Emden (1907), such that the interest of the LE derived from its nonlinearity and singular behavior at the origin. The point $x_0$ is called ordinary point (or regular point) of the dynamic equation \eqref{lane-emden2} if the coefficients of $x,x'$ are analytic in an interval about $x_0.$ Otherwise, it is called singular point. In solving singular boundary value problems (BVPs) some numerical techniques are based on the idea of replacing a two-point BVP by two suitable initial value problem \cite{Herron, KouroshAmin, Russ}. In this paper we adopt such idea (called the shooting method) to study dynamical models that play an essential role in the theory of star structure and evolutions, thermodynamics, and astrophysics (see, e.g., \cite{ColliTee}). Equation \eqref{lane-emden1} describes and models the mechanical structure of a spherical body of gas such as a self-gravitating star and also appeared in the study of stellar dynamics (see, for instance \cite{Chandra,Davis} and the references therein).
	The solutions to the LE, which are known as polytropes, are functions of density versus the radius expressed by $x(t)$ in \eqref{lane-emden2}. The index $n$ determines the order of that solution. Nonlinear singular LE equations can be formulated as 
	\begin{equation}\label{lane-emden1}
		\frac{1}{t^2}\frac{d}{dt}(t^2\frac{d x}{d t})+x^n=0 
	\end{equation} or,  
	\begin{equation}\label{lane-emden2}
		x''(t)+\frac{2}{t}x'(t)+[x(t)]^n=0, \ n\geq 0
	\end{equation} subject to $$x(0)=1 \ , \ x'(0)=0.$$
	The dynamical system model \eqref{lane-emden2} along with initial conditions form a special type of initial value problems (IVP) for which it has several applications in the fields of celestial mechanics, quantum physics and astrophysics \cite{Bile,Datt,Herron,sand}. 
	The following figure is a motivation example shows finite solutions of Lane-Emden equation for the value of $n$ in equation \eqref{lane-emden1} or \eqref{lane-emden2} given by $n=0,1 ,2 ,3 ,4 ,5, 6.$
	
	\begin{center}
		\includegraphics[width=.55\textwidth]{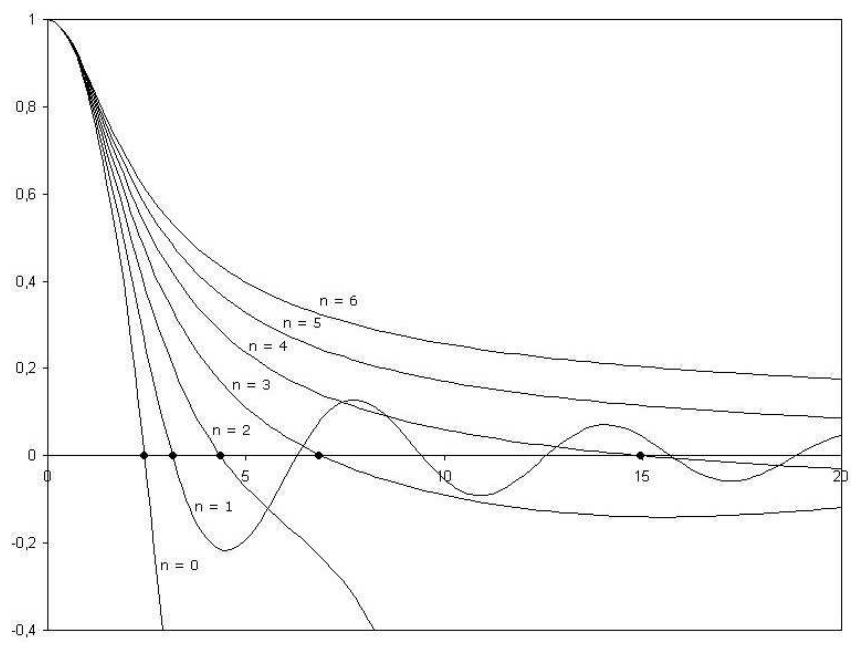}
	\end{center}
	
	For some special cases when $n=0,1,5$ exact analytical solutions were obtained by Chandrasekhar \cite{Chandra}, while for all other values of $n$ approximate analytical methods were obtained such as: the Adomian decomposition method \cite{Motsa,Wazwaz}, homotopy analysis method \cite{Bataineh}, power series expansions \cite{Karimi}, variational method \cite{He}, and linearization techniques \cite{Ramos} (provide accurate closed-form solutions around the singularity.). 
	Numerical discretization for equation \eqref{lane-emden1} has been the object of several studies in the last decades (see, e.g., \cite{Jedara,Ali,Asadpour, Bile, Datt, Kadal, KouroshAmin, Russ, sand} and the references therein). In \cite{Karimi}, the authors presented numerical method for solving singular IVPs by converting Lane-Emden-type equation \eqref{lane-emden1} to an integral operator form then rewriting the acquired Voltera integral equation in terms of a power series. Ramos \cite{Ramos} applied linearization method for the numerical solution of singular initial value problems of linear and nonlinear, homogeneous and nonhomogeneous second-order dynamic equations. 
	Russell and Shampine in \cite{Russ} discussed the solution of the singular nonlinear BVP for certain dynamical systems in the context of analytical geometry and symmetry as follows
	\begin{equation}
		\label{RussDE}
		x''(t)+\frac{k}{t}x'(t)+g(t,x)=0, \ \ \text{where} \ k=0,1,2,
	\end{equation} 
	and with boundary conditions $x'(0)=0$ (\text{or equivalently} $x(0)$ \text{is finite}),\ $x(b)=\lambda,$\ \text{for some scalar}\ $\lambda,$ and the convergence is uniform over the interval $[0,1]$.
	Biles et al. in \cite{Bile}, have considered an initial value problem for Lane-Emden type of the form 
	\begin{equation}
		x''(t)+p(t)x'(t)+q(t,x(t))=0, \ t>0\ \ \text{with}\ \ x(0)=a , x'(0)=0
		\label{Biles}
	\end{equation} 
	where $a\in\mathbb{R}$ and $p(t)$ may be singular at $t=0.$ They introduced the following definition and theorem, respectively; where the theorem gives the conditions of existence and uniqueness of solution of second-order linear BVPs. 
	\begin{definition}[\cite{Bile}] $x$ is a solution of the above equation \eqref{Biles} if and only if there exist some $T>0,$\  such that $x,x'$ are absolutely continuous on [0,T]. 
	\end{definition}
	\begin{theorem}[\cite{Bile}] Suppose in the above equation \eqref{Biles} $p$ is measurable on [0,1], non-negative on (0,1] and $\int_0^1 sp(s) ds$ is finite, and $q$ is bounded. Specifically, suppose there exist $\alpha,\beta$ with $\alpha<a<\beta$ and $K>0$ such that: 
		\begin{enumerate}
			\item for each $t\in[0,1]$, $q \in C\big([\alpha,\beta]\big)$; and $q$ is Lipschitz in $y$ on $[\alpha,\beta]$
			\item for each $x \in [\alpha,\beta], \ q$ is measurable on [0,1]; and 
			\item $\sup\limits_{(t,x) \in [0,1]\times[\alpha,\beta]}$ $|q(t,x)|\leq K$. 
			\item Suppose that $q$ is Lipschitz in $y$ on $[\alpha,\beta]$. 
			Then equation \eqref{Biles} has a unique solution.
		\end{enumerate}
	\end{theorem}
	Our paper is organized in the following fashion. In section 2, we provide some necessary notations and essential background. In section $3$ we present the second-order dynamical system of Lane-Emden type, and the BVP is transformed to IVP by shooting method. Then applying Euler's method on the resulted initial value problem to get approximations for the solution of the LE. The convergence results and error estimation are analyzed in section 4. Finally, numerical examples are provided to demonstrate the validity and efficiency of the proposed technique.

	\section{Preliminaries}
	In this section we introduce some basic definitions and conventional notations. Let $C^1(I)$ be the space of all continuously differentiable functions defined on an interval $I$. A set $D$ in the Euclidean space $\mathbb{R}^n$ is compact set if and only if it is closed and bounded set. The basic space used throughout this paper is the space of continuous functions $C[0,1]$ on the compact set 
	$[0,1]$ with the associated norm (distance) function defined by, $$\|x\|=\max_{0\leq t\leq1}|x(t)|.$$ 
	Define a continuous function $f:D\to\mathbb{R}^n$ where $D$ is an open subset of $\mathbb{R}^{n+1},$ and consider the dynamical system 
\begin{equation}
	\label{ode}
	\dot{x}(t)=f(t,x)\ ,\hspace{.7cm} x(t_0)=x_0.
\end{equation}
Given $(t_0,x_0) \in D$, a continuous function $x(t)$ in an open interval $(a,b)$ containing $t_0$ is a solution of the IVP \eqref{ode} if and only if $$x(t)=x_0+ \int_{t_0}^{t}f(s,x(s)) \ ds$$ for every $t \in (a,b)$.
Conventionally, most of dynamic evolution equations of this type \eqref{ode} arising in application-driven aspects cannot be solved algebraically or exactly, but they can be investigated qualitatively without knowing the exact solutions. As we know, qualitative approaches are not very accurate, hence, an approximate solution (more accurate) of this dynamic equation \eqref{ode} can be obtained by successive approximations methods. We say $f$ is differentiable function if its graph $\Gph f:=\{(t,x(t));\ t\in (a,b)\}$ has a slope defined at every point $t$ in the interval $(a,b)$. 
\begin{definition} 
	Let $D$ be a nonempty set. Suppose there is a function $f$ from $D$ to itself, and $0\leq L<1,$ where $L$ is free of $x$ and $y$. If for any two points $x,y\in D$ we have $$|f(x)-f(y)|\leq L|x-y| \ , \ \forall \ x,y\in D,$$ then $f$ is called a contraction. The smallest such value of $L$ is called the Lipschitz constant of $f$, and $f$ is then called a Lipschitz function.
\end{definition}
\begin{definition}
	A function $f:D\subset\RR^{n+1}\to\RR^n$ is said to be \textit{locally Lipschitz} in $x$ if for each compact set contained in $D$, and each $x, y \in D,$ there exists $L>0$ such that $$\|f(t,x)-f(t,y)\|\leq L\|x-y\|.$$
\end{definition}
In particular, all $C^1$ functions are locally Lipschitz. The following two theorems address existence and uniquness of solutions to any IVP. 
\begin{definition} A sequence ${x_n(t)}$ of functions in $C[a, b]$ converges uniformly to a function $x(t)\in C[a, b]$ if and only if $\displaystyle\lim_{n\to\infty}\|x_n-x\|=0.$
\end{definition}
\begin{theorem}\label{PicLind}
	(Picard-Lindelof theorem). If the function $f : D \rightarrow \mathbb{R}^n$ is continuous and locally Lipschitz in $x$ in an open set $D \subset \mathbb{R}^{n+1} $ , then for each $(t_0 , x_0) \in D$ there exists a unique solution of the initial value problem in some open interval containing $t_0$.
\end{theorem}
%
\begin{theorem}[\cite{Burd}]
	Assume $\hat{a}(t,x(t),x'(t))\in C([0,1]\times \mathbb{R}\times \mathbb{R})$ and $\hat{a},\frac{\partial\hat{a}}{\partial x}, \frac{\partial\hat{a}}{\partial x'}\in C([0,1]\times \mathbb{R}\times \mathbb{R})$.
	If $\frac{\partial\hat{a}}{\partial x}>0$ and there exist $M>0$ such that $\left| \frac{\partial\hat{a}}{\partial x'} \right|<M, \forall (t,x,x') \in [0,1] \times \mathbb{R} \times \mathbb{R}$, 
	then the BVP 
	\begin{equation}\label{Burd}
		\frac{d^2 x}{d t^2}=\hat{a}(t,x,x')
	\end{equation}
	with $$x(0)=\alpha \ ,\ x(1)=\beta,$$  has a unique solution $x=x(t)$.
\end{theorem}
To better understand the theorem we illustrate it by giving an example on the interval $[1, 2]$ instead of $[0, 1]$: Consider the BVP, $$ x''(t)+\sin x'+e^{-tx}=0$$with $x(1)=x(2)=0$ and $t\in[1,2]$. Now apply the theorem to $$ x''(t)=-\sin x'-e^{-tx}=\hat{a}(t,x,x').$$ Since $q(t,x(t))= \frac{\partial \hat{a}}{\partial x}=te^{-tx}>0$, $\forall t>0,$ and $\left| p(t)=\frac{\partial \hat{a}}{\partial x'} \right|=\left|-\cos x'\right| \leq 1=M$, then the condition is satisfied and the BVP has a unique solution.
Now reader might ask how can we apply this Theorem to Lane-Emden equation. Theorem 2.2 can be simplified by taking into account that the functions $\frac{\sin x'}{x'}$ and $e^{-tx}$ are continuous on the interval $(0,\infty)$ to assure the differential equation is linear.
%

\section{\textbf{Computational Mehtods For Dynamical systems}}
In this section, we start by presenting the methods (shooting to transform from BVP to IVP, and Euler's for regular singularity in the drift term) and apply them on the second order singular dynamical system.
\subsection{Shooting method}
The shooting method treats the two-point BVP as an IVP. The idea basically, is to write the BVP in a vector form and begin the solution at one end of the BVP, and then "shooting" to the other end with any IVP solver, such as; Runge-Kutta method or multistep method for linear case and Secant method or Newton's method for nonlinear case, until the boundary condition at the other end converges to its correct value. To be precise, the ordinary differential equation of second order, associated with its initial conditions must normally be written as a system of first order equation before it can be solved by standard numerical methods. Next figure shows graphically the mechanism of the shooting. 
\begin{center}
	\includegraphics[width=.5\textwidth]{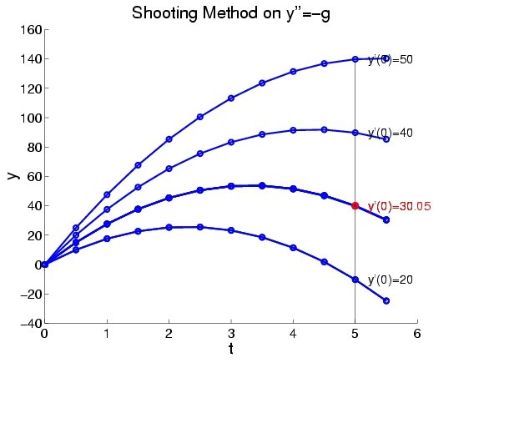}
\end{center}
Roughly speaking, we 'shoot' out trajectories in different directions until we find a trajectory that has the desired boundary value. 
The drawback of the method is that it is not as robust as those used to solve BVPs such as finite difference or collocation methods presented in \cite{Jedara,Russ}, and there is no guarantee of convergence.
Shooting method can be used widely for solving a BVP by reducing it to an associated IVP, and is valid for both linear (also called chasing method) and non linear BVPs, by \cite{Othmar}, 
\begin{equation}\label{Othmar}
	\frac{d^2 x}{d t^2}=\hat{a}(t,x(t),x'(t)),\  x(t_0)=x_0,\ x(t_1)=x_1.
\end{equation} 
Next theorem provides existence and uniqueness to the BVP's solution. 

\begin{theorem}
	Define a set $D:=\{(t,x,x')\in [a,b]\times\RR\times\RR\}$, and assume $f$ is continuous function on $D$ such that it satisfies the BVP: 
	\begin{equation}
		\label{bvp}
		\left\{\begin{array}{l}
			x''(t)=f(t,x,x')\\ x(a)=\alpha \\ x'(b)=\beta.
		\end{array} \right.
	\end{equation}
	Suppose that $f_x$ and $f_{x'}$ are continuous on the same set $D$. If \\
	{\bf (i)} $f_x(t,x,x')>0$ for all values, and \\
	{\bf (ii)} There exists $M>0$ such that $$|f_{x'}(t,x,x')|\leq M , \quad\forall (t,x,x')\in D$$
	then the BVP \eqref{bvp} has a unique solution.
\end{theorem}  
A special case of this theorem is the following corollary, \emph{i.e.}, when the right hand side of \eqref{bvp} is linear. For linear Lane-Emden equations, one can use Frobenius method to determine the analytical solutions of \eqref{lane-emden1} near the singularity, see, for instance, \cite{Ramos}.
\begin{corollary} 
	Consider \eqref{bvp} given by  
	\begin{equation}\label{Ramos}
		x''(t)=p(t)x'+q(t)x+r(t),
	\end{equation}and the time-dependent coefficients $\ p(t) \ , \ q(t) \ , \ r(t) \ $ are continuous functions on the domain $[a,b]$ and further $q(t)>0$, then the BVP \eqref{bvp} has a unique solution.  
\end{corollary}
\begin{proof}
	We need to consider two cases: {\bf (i)} When equation \eqref{Ramos} given with boundary conditions $x(a)=\alpha \ , \ x'(a)=0,$ has a unique solution $x_1(t).$ \ {\bf (ii)} When equation \eqref{Ramos} with $r(t)=0, \ x(a)=0 \ , \ x'(a)=1,$ has a unique solution $x_2(t).$ 
	Therefore, one can easily check that the linear combination $\hat{x}(t)=x_1(t)+\displaystyle\frac{\alpha -x_1(b)}{x_2(b)}x_2(t)$ is the unique solution to \eqref{Ramos}, and hence to \eqref{bvp} due to the existence and uniqueness guaranteed by Picard-Lindelof theorem \eqref{PicLind}.
\end{proof}

\subsection{Euler's Method}
Euler's method is a numerical approach for solving (iteratively) initial value problems, as follows: We divide the time interval $[t_0,T]$ into $N$ equal subintervals, each of length $h=\Delta t=t_{n+1}-t_n$, for $n\geq 0$, \ and start by initial value $x(0)$ then move forward using the step size towards $x(T)$, that is, given the second-order ordinary differential equation \eqref{Othmar}, converting it into two first-order dynamic equations (i.e., dynamical system). Discretize the interval $[t_0,T]$ into subintervals, and by assuming $y_n$ the approximation to $x(t_n)$ and $v_n$ the approximation to $u(t_n).$ Euler's method is then can be expanded, as a two-terms truncated Taylor series, by the following Euler's method for solving a second-order differential equation is given by:
\begin{center}\fbox{\begin{minipage}[b]{\textwidth}
			\noindent{\bf Forward Euler's Algorithm.} 
			\medskip
			
			\noindent{\bf Step 0. (Initialization):}
			Take
			\begin{equation*}
				t_0,\ x_0\in\RR,\quad\mbox{and \ step size}\quad h=\frac{T-t_0}{N}, \ n\geq 0.
			\end{equation*}
			\smallskip
			\noindent {\bf Step 1. (Forward step):} Given $t_n\ , \ y_n\ , \ v_n$ define
			\begin{align*}
				t_{n+1} &= t_n + h, \\
				y_{n+1} &= y_n + h \cdot v_n, \\
				v_{n+1} &= v_n + h \cdot\hat{a}(t_n, y_n, v_n),
			\end{align*}
			
			\noindent {\bf Stopping Criterion:} If $v_{n+1}=v_n$ then stop.
\end{minipage}}\end{center}
\vspace{.4cm}

The local error at every step is proportional to the square of the step size $h$ and the global error at a given time is proportional to $h$. Moreover, the order of the global error can be calculated from the order of the local error ( i.e. by summing up the local error). We can understand Euler's method by appealing the idea that some differential equations provide us with the slope at all points of the function , while an initial value provides a point on the function. Using this information we can approximate the function with a tangent line at the initial point. It is known that the tangent line is only a good approximation over a small interval. When moving to a new point, we can construct an approximate tangent line, using the actual slope of the function, and an approximation to the value of the function at the tangency point. Repeating this manner, we eventually construct a piecewise-linear approximation to the solution of the differential equation. Moreover, this approximation can be seen as a discrete function and to make it a continuous function, we interpolate (linearly) between each pair of these points.

In the following, we study and analyse the Lane-Emden-type equation with an endpoint singularity in terms of the independent variable which has the form  
\begin{equation}\frac{d^2 x}{d t^2}=\frac{-a(t,x)}{1-t}\frac{d x}{d t}+g(t,x) = \hat{a}(t,x,x')
	\label{Lane}
\end{equation}
where $\hat{a}(t,x(t),x'(t)): [0,1)\times\RR\times\RR\to\RR$, and the Lipschitz functions $a(t,x), g(t,x)\in C^1([0,1)\times\RR)$, for all $0 \leq t<1$. 
At $t=1$, the $\displaystyle\frac{-a(t,x)}{1-t}$ term is singular, but symmetry implies the boundary condition $x'(0)=0.$ With this boundary condition, the term $\displaystyle\frac{-a(t,x)}{1-t}\frac{dx}{dt}$ is well defined as $t\to 1.$
The solution of \eqref{Lane} can be given by the system:
\begin{equation}\label{dx'dx} 
	\begin{aligned}
		dx'&=\hat{a}(t,x(t),x'(t))dt \\ dx&=x'dt 
	\end{aligned}
\end{equation}
Define $ x_t:=x(t), x'_t:=x'(t)$. By the fundamental theorem of calculus and provided that all integrals are exist (finite), we notice that equation \eqref{dx'dx} is equivalent to the nonlinear system of integral equations:
\begin{equation}\label{integralnotation}
	\begin{aligned} 
		x'_t &= x'_{t_n} +\int_{t_n}^t \hat{a}(s,x_s,x'_s)\,ds \\ 
		x_t &= x_{t_n} +\int_{t_n}^t x'_s\,ds .
	\end{aligned}
\end{equation}
Where 
$$ 0 = t_0 < t_1 < t_2 < ... < 1 .$$
%
Expanding the integrands in \eqref{integralnotation} so we have:
\begin{align*}
	x'_t &= x'_{t_n}+\int_{t_n}^t \Big[\hat{a}(t_n, x_{t_n}, x'_{t_n})
	+\int_{t_n}^s\Big[\frac{\partial\hat{a}}{\partial t}(u,x_u,x'_u)\\
	& \quad+\frac{\partial\hat{a}}{\partial x}(u,x_u,x'_u)x'_u+\frac{\partial \hat{a}}{\partial x'}(u,x_u,x'_u)\hat{a}(u,x_u,x'_u)\Big]du\Big]ds \\ 
	x_t &= x_{t_n}+\int_{t_n}^t\big[x'_{t_n}+\int_{t_n}^s\hat{a}(u,x_u,x'_u)du\big]ds.
\end{align*}
Or in the equivalent form,
\begin{align*}
	x'_t &= x'_{t_n}+\hat{a}(t_n,x(t_n),x'(t_n))(t-t_n)\\
	& \qquad+\int_{t_n}^t \!\int_{t_n}^s (\frac{\partial\hat{a}}{\partial t}+\frac{\partial\hat{a}}{\partial x} x'_u+\frac{\partial\hat{a}}{\partial x'}\hat{a})(u,x_u,x'_u)\,du\,ds\\ 
	x_t &= x_{t_n} + x'_{t_n}(t-t_n) + \int_{t_n}^t \!\int_{t_n}^s \hat{a}(u,x_u,x'_u)\,du\,ds 
\end{align*}
For simplicity we assume
\begin{align*}
	L^{(1)}_n &=\int_{t_n}^t\!\int_{t_n}^s \big(\frac{\partial\hat{a}}{\partial t}+\frac{\partial\hat{a}}{\partial x}x'_u+\frac{\partial\hat{a}}{\partial x'}\hat{a})(u,x_u,x'_u)\,du\,ds,  
	\\
	L^{(2)}_n &=\int_{t_n}^t\!\int_{t_n}^s \hat{a}(u,x_u,x'_u)\,du\,ds.
\end{align*}
%
%
%
\newline
Thus the system becomes,
\begin{equation}
	\begin{aligned}
		x'_{t_{n+1}}&=x'_{t_n}+\hat{a}(t_n, x(t_n), x'(t_n))(h_{n+1})+ L_n^{(1)} \\ x_{t_{n+1}}&=x_{t_n}+x'_{t_n} h_{n+1}+L_n^{(2)}  
	\end{aligned}
\end{equation} where $h_{n+1}=t_{n+1}-t_n.$
\newline 

In order to estimate the error, we need to find a bound for the integrands in $L^{(1)}_n$ and $L^{(2)}_n$. The double integrals in both $L^{(1)} , L^{(2)}$ yield the local truncation error, if we define the numerical value by:
\begin{equation}
	\begin{aligned}\label{numerical value y}
		y'_{n+1}&=y'_n + \hat{a}(t_n,y_n,y'_n)h_{n+1} \\ 
		y_{n+1}&=y_n + y'_n h_{n+1}.
	\end{aligned}
\end{equation}
where $h_{n+1}=t_{n+1}-t_n.$


\section{\textbf{Discretization and Convergence Analysis}} 

Consider a sequences of times  
$ 0 = t_0 < t_1 < t_2 < ... < 1 ,$
and the corresponding step sizes $h_n=t_n - t_{n-1}.$
Define $x_n = x(t_n)$ and $x'_n=x'(t_n)$ where $(x(t), x'(t))$ is a solution of (5). Writing (8) in the form: 
\begin{equation}
	\begin{aligned}
		x'_{n+1} &=x'_n +\hat{a}(t_n, x_n, x'_n)(h_{n+1})+ L_n^{(1)} \\ 
		x_{n+1}  &=x_n +x'_{n} h_{n+1}+L_n^{(2)}  
	\end{aligned}
\end{equation}
Use $y_n$ as defined in (9) and let $\epsilon_i=x_i-y_i \,,\epsilon'_i=x'_i(t)-y'_i(t), \forall i.$ So we have \\ 
\begin{align*} 
	\epsilon'_{n+1}&=\epsilon'_n +\big[\hat{a}(t_n,x_n, x'_n)-\hat{a}(t_n,y_n,y'_n)\big](h_{n+1})+ L_n^{(1)} \\ 
	\epsilon_{n+1}&=\epsilon_n +\epsilon'_n h_{n+1}+L_n^{(2)}
\end{align*}
\\ 
By using the inequality $(x+y)^2\leq2x^2+2y^2$, the error can be estimated as, 
\begin{equation}\label{epsiloninequality}
	\begin{split}
		(\epsilon'_{n+1})^2\leq(\epsilon'_n)^2+2\big[\hat{a}(t_n,x_n,x'_n)-\hat{a}(t_n,y_n,y'_n)\big]^2(h_{n+1})^2 +2(L_n^{(1)})^2 \\ +2\epsilon'_n\Big(\hat{a}(t_n,x_n,x'_n)-\hat{a}(t_n,y_n,y'_n)\Big)h_{n+1}+2\epsilon'_nL_n^{(1)} \\
		(\epsilon_{n+1})^2\leq(\epsilon_n)^2+2(\epsilon'_n)^2(h_{n+1})^2+2(L_n^{(2)})^2+2\epsilon_n\epsilon'_n h_{n+1} +2\epsilon_n L_n^{(2)}.
	\end{split}
\end{equation}

Next, we introduce some assumptions on the functions $a(t,x(t)), g(t,x(t))$ and their partial derivatives for $t \in [0,1), x \in \mathbb{R}$ . But before that we remind ourselves of the value of $\hat{a}$ from section 3, 
$$ \hat{a}(t,x(t),x'(t))=\frac{-a(t,x(t))}{1-t}\frac{dx}{dt}+g(t,x(t)). $$
Also, for any $ T_1 , T_2 \in [0,1)$ the Lipschitz conditions are:
$$|a(t,x)-a(t,y)| \leq T_1 |x-y| \,,\, |g(t,x)-g(t,y)| \leq T_2 |x-y|.$$
Our required bounds explicitly are:
$$|a(t,x(t))| \leq C_0 \,,\, |g(t,x(t))| \leq C_3 .$$ 
The partial derivatives bounds are:
\begin{subequations}
	\begin{align}\Big|\frac{\partial a}{\partial t} (t,x(t))\Big| & =|a_1(t,x(t))| \leq C_1,\\ 
		\Big|\frac{\partial a}{\partial x}(t,x(t))\Big| & =|a_2(t,x(t))| \leq C_2,\\ 
		\Big|\frac{\partial g}{\partial t}(t,x(t))\Big| & =|g_1(t,x(t))| \leq C_4,\\
		\Big|\frac{\partial g}{\partial x}(t,x(t))\Big| & =|g_2(t,x(t))| \leq C_5 .
	\end{align}
\end{subequations}

This final bound applies along the path
$$|x'(t)|\leq A_1.$$ 
Taking the difference between the computed and the exact values of $\hat{a},$ 
\begin{equation}\label{differencea^}
	\begin{aligned}
		\big|\hat{a}(t,x,x')-\hat{a}(t,y,y')\big|
		&= \Big|\frac{-a(t,x)}{1-t}x'+g(t,x)+\frac{a(t,y)}{1-t}y'-g(t,y)\Big|\\ 
		&\leq \bigg|\frac{a(t,y)y'-a(t,x)x'}{1-t}\bigg|+\Big|g(t,x)-g(t,y)\Big|.
	\end{aligned}
\end{equation}
\\
By adding and subtracting the required terms, we have
\begin{align*}
	\big|a(t,y)y' - a(t,x)x'\big|
	&= \big|a(t,x)(y'-x')+x'(a(t,y)-a(t,x))+(a(t,y)-a(t,x))(y'-x')\big| \\
	&\leq C_0\,|y'-x'|+A_1T_1\,|y-x|+T_1\,|y-x|.|y'-x'|. 
\end{align*}
Thus, the difference \ref{differencea^} becomes, 
$$|\hat{a}(t_n,x_n,x'_n)-\hat{a}(t_n,y_n,y'_n)| \leq \frac{C_0|\epsilon'_n|}{1-t}+\frac{A_1T_1|\epsilon_n|}{1-t}+\frac{T_1|\epsilon_n|\,|\epsilon'_n|}{1-t}+T_2|\epsilon_n|.$$
Note that,
\begin{align*}
	\frac{\partial\hat{a}}{\partial t}
	&= \hat{a_1}(t,x,x')\\ 
	&=\frac{-a_1(t,x)x'}{1-t}-\frac{a(t,x)x'}{(1-t)^2}+g_1(t,x),\\
	\frac{\partial\hat{a}}{\partial x}x'
	&=\frac{-a_2(t,x)}{1-t}(x')^2+g_2(t,x)x',\\ 
	\frac{\partial \hat{a}}{\partial x'}\hat{a}
	&=\frac{a^2(t,x)}{(1-t)^2}x'-\frac{a(t,x)g(t,x)}{1-t}.
\end{align*}
We now apply a very well known result from functional analysis, Cauchy-Schwarz inequality twice on $L^{(1)} \text{and}  L^{(2)}:$
\begin{align*}
	\big(L_N^{(1)}\big)^2
	&= \bigg(\int_{t_n}^{t_{n+1}}\!\int_{t_n}^{t}(\frac{\partial\hat{a}}{\partial t}+\frac{\partial\hat{a}}{\partial x}x'+\frac{\partial\hat{a}}{\partial x'}\hat{a})\,ds\,dt\bigg)^2 \\
	&\leq h^2_{n+1}\int_{t_n}^{t_{n+1}}\!\int_{t_n}^{t}\Big(\frac{\partial\hat{a}}{\partial t}+\frac{\partial\hat{a}}{\partial x}x'+\frac{\partial\hat{a}}{\partial x'}\hat{a}\Big)^2\,ds\,dt \\
	&\leq h^2_{n+1}\int_{t_n}^{t_{n+1}}\!\int_{t_n}^{t}\Bigg[3\Big(\frac{\partial\hat{a}}{\partial t}\Big)^2+3\Big(\frac{\partial\hat{a}}{\partial x}x'\Big)^2+3\Big(\frac{\partial\hat{a}}{\partial x'}\hat{a}\Big)^2\Bigg]\,ds\,dt \\
	& \leq 3h^2_{n+1} \int_{t_n}^{t_{n+1}}\!\int_{t_n}^{t} \bigg( \frac{3C_1^2A_1^2}{(1-t)^2}+\frac{3C_0^2A_1^2}{(1-t)^4}+3C_4^2 \\
	&+\frac{2C_2^2A_1^4}{(1-t)^2} + 2C_5^2A_1^2 + 2\frac{C_0^4A_1^2}{(1-t)^4} + 2\frac{C_0^2C_3^2}{(1-t)^2}\bigg)\,ds\,dt \\ 
	&\leq D_1\frac{h_{n+1}^4}{(1-t_{n+1})^4}.
\end{align*} for some Constant $D_1,$ which does not depend on $h_{n+1}$ and $n$. 
\begin{align*}
	\big(L_N^{(2)}\big)^2
	&= \bigg(\int_{t_n}^{t_{n+1}}\!\int_{t_n}^{t}(\frac{-a(t,x)}{1-s}+g(t,x))\,ds\,dt\bigg)^2 \\ &\leq h_{n+1}^2 \int_{t_n}^{t_{n+1}}\!\int_{t_n}^{t}\Big(\frac{-a(t,x)}{1-s}+g(t,x)\Big)^2\,ds\,dt \\  &\leq h_{n+1}^2 \int_{t_n}^{t_{n+1}}\!\int_{t_n}^{t}2\frac{a^2(t,x)}{(1-s)^2}\,ds\,dt\,+h_{n+1}^2\int_{t_n}^{t_{n+1}} \!\int_{t_n}^{t}2g^2(t,x)\,ds\,dt \\ &\leq 2h_{n+1}^2\Big( \int_{t_n}^{t_{n+1}}\!\int_{t_n}^{t}\frac{C_0^2}{(1-s)^2}\,ds\,dt\,+\int_{t_n}^{t_{n+1}}\! \int_{t_n}^{t}C_3^2\,ds\,dt\Big) \\ &= 2h_{n+1}^2\big(C_0^2 \int_{t_n}^{t_{n+1}} \frac{-1}{1-s}\,dt+C_3^2\int_{t_n}^{t_{n+1}}(t-t_n)\,dt\big) \\
	&\leq \frac{2h_{n+1}^4C_0^2}{(1-{t_{n+1}})^2}+\frac{C_3^2}{2}h_{n+1}^4 \\ &\leq D_2\frac{h_{n+1}^4}{(1-t_{n+1})^2}.
\end{align*}
where $D_2$ is independent of $n$ and $h_{n+1}$. 

To avoid the singularity and produce a better estimation to test the efficiency of the algorithm, we introduce a variable step size by fixing 
$\hat{h}>0$ 
and then defining step size $h_n$ and node points $t_n$ using $\hat{h}$:
$$\hat{h}=\frac{h_{n+1}}{1-t_{n+1}}$$ 
or,
\begin{equation}\label{stepsize}
	t_{n+1}=t_n+\hat{h}(1-t_{n+1}).
\end{equation}

In the process of estimating the global error, we need to use the following two fundamental lemmas:
\begin{lemma}\label{binomialexponential}
	For all $x\geq -1$, and any $m>0$, we have $0\leq (1+x)^m\leq e^{mx}.$
\end{lemma}
The proof of this result follows by applying Taylor's theorem with $f(x)=e^x, \ x_0=0$, and $n=1.$
\begin{lemma}\label{mainlemma}
	if $M_1\geq-1$ and $M_2\geq0$ are real numbers and $\left\{a_n\right\}^{N}_{n=0}$is a sequence with $a_0 \geq0$ such that \begin{equation}  a_{n+1}\leq(1+M_1)a_n+M_2 , \ \forall n=0, 1, 2, \dots, N-1,
		\label{14}
	\end{equation} then, 
	\begin{equation}\label{upperbound-for-an} a_{n+1}\leq e^{(N+1)M_1}\bigg(\frac{M_2}{M_1}+a_0\bigg)-\frac{M_2}{M_1}, \ \forall n=0, 1, 2, \dots ,N-1. \end{equation}
	\label{4.1} 
\end{lemma}
\begin{proof}
	Fix a positive integer $n$, then \eqref{14} can be written as 
	\begin{align*}
		a_{n+1} &\leq (1+M_1)a_{n}+M_2 \\ &\leq (1+M_1)\big[ (1+M_1)a_{n-1}+M_2\big] a_{n}+M_2 \\ & \ \ \vdots \\ &\leq (1+M_1)^{n+1}a_0+\big[1+(1+M_1)+(1+M_1)^2+\dots+(1+M_1)^n\big]M_2 \\ &\leq (1+M_1)^{n+1}a_0+\Big[\sum_{j=0}^n(1+M_1)^j\Big]M_2 \\
		&\leq (1+M_1)^{n+1}a_0+\Big[\displaystyle\frac{1-(1+M_1)^{n+1}}{1-(1+M_1)}\Big]M_2 \hspace{2.4cm} \text{(sum of geometric series)} \\ &\leq (1+M_1)^{n+1}a_0+\Big[(1+M_1)^{n+1}-1\Big]\frac{M_2}{M_1} \\ &\leq (1+M_1)^{n+1}\Big(a_0+\frac{M_2}{M_1}\Big)-\frac{M_2}{M_1}.
	\end{align*}
	By Lemma \ref{binomialexponential}, equation \eqref{upperbound-for-an} follows, i.e., 
	\begin{align*}
		a_{n+1}\leq e^{(1+N)M_1}\Big(a_0+\frac{M_2}{M_1}\Big)-\frac{M_2}{M_1}.
	\end{align*}
\end{proof}
Now if we add the two inequalities in (11) together, we will have
\begin{equation}
	\begin{split}
		(\epsilon'_{n+1})^2 + (\epsilon_{n+1})^2 
		&\leq (\epsilon'_n)^2+(\epsilon_n)^2+2h_{n+1}^2(\epsilon'_n)^2+2\big[\hat{a}(t_n,x_n,x'_n)-\hat{a}(t_n,y_n,y'_n)\big]^2 h_{n+1}^2 \\
		&+ 2(L_n^{(1)})^2+2(L_n^{(2)})^2+2\epsilon_n\epsilon'_n h_{n+1} +2\epsilon_n L_n^{(2)} \\
		&+ 2\epsilon'_n\Bigg(\Big(\hat{a}(t_n,x_n,x'_n)-\hat{a}(t_n,y_n,y'_n)\Big)h_{n+1}+L_n^{(1)}\Bigg) \\ 
		&\leq (\epsilon'_n)^2+(\epsilon_n)^2+2h_{n+1}^2(\epsilon'_n)^2+8C_0^2(\epsilon'_n)^2\Big(\frac{h_{n+1}}{1-t_{n+1}}\Big)^2 \\ 
		&+ 8A_1^2T_1^2\epsilon_n^2\Big(\frac{h_{n+1}}{1-t_{n+1}}\Big)^2+8T_1^2\epsilon_n^2(\epsilon'_n)^2\Big(\frac{h_{n+1}}{1-t_{n+1}}\Big)^2+8T_2^2\epsilon_n^2h_{n+1}^2 \\
		&+ 2D_1\Big(\frac{h_{n+1}}{1-t_{n+1}}\Big)^4+2D_2\Big(\frac{h_{n+1}}{1-t_{n+1}}\Big)^2h_{n+1}^2+2\epsilon_n\epsilon'_nh_{n+1} \\
		&+ 2\epsilon'_n\sqrt{D_1}\Big(\frac{h_{n+1}}{1-t_{n+1}}\Big)^2+2\epsilon_n\sqrt{D_2}\Big(\frac{h_{n+1}}{1-t_{n+1}}\Big)h_{n+1} \\
		&+ 2(\epsilon'_n)^2C_0\Big(\frac{h_{n+1}}{1-t_{n+1}}\Big)+2A_1T_1\epsilon_n\epsilon'_n\Big(\frac{h_{n+1}}{1-t_{n+1}}\Big) \\
		&+ 2T_1\epsilon_n(\epsilon'_n)^2\Big(\frac{h_{n+1}}{1-t_{n+1}}\Big)+2T_2\epsilon_n\epsilon'_nh_{n+1} \\
		&\leq \bigg[ K_1 h_{n+1}^2+K_2 \Big(\frac{h_{n+1}}{1-t_{n+1}}\Big)^2+h_{n+1}+K_3 \Big(\frac{h_{n+1}}{1-t_{n+1}}\Big) \bigg] ||\epsilon_n||^2 \\
		&+ 2D_1\Big(\frac{h_{n+1}}{1-t_{n+1}}\Big)^4+2D_2\Big(\frac{h_{n+1}}{1-t_{n+1}}\Big)^4+ K_4\epsilon_n(\epsilon'_n)^2\Big(\frac{h_{n+1}}{1-t_{n+1}}\Big) \\
		&+ 2\Big(\frac{h_{n+1}}{1-t_{n+1}}\Big)^2\Big[\sqrt{D_1}+\sqrt{D_2}\Big]\sqrt{(\epsilon'_n)^2+(\epsilon_n)^2}+K_5^2\epsilon_n^2(\epsilon'_n)^2\Big(\frac{h_{n+1}}{1-t_{n+1}}\Big)^2. 
	\end{split}
\end{equation}
Using the definition of the norm  $\|\epsilon_n\|=\sqrt{(\epsilon'_n)^2+(\epsilon_n)^2}$ , then system (13) can be simplified as 
$$(\epsilon'_{n+1})^2+(\epsilon_{n+1})^2\leq (\epsilon'_n)^2+(\epsilon_n)^2+m_1(\hat{h})\Big[(\epsilon'_n)^2+(\epsilon_n)^2\Big]+m_2(\hat{h})^3$$
where $m_1$ and $m_2$ are independent constants of $h_{n+1}$ and $t_{n+1}.$
Now we apply Lemma \ref{mainlemma} for $a_n=\|\epsilon_n\|^2$,\ followed by a foundation for the step size order, with $M_1=1+m_1(\hat{h})$ and $M_2=m_2(\hat{h})^3$ such that if 
\begin{align*}
	\|\epsilon_{n+1}\|^2 \leq \|\epsilon_n\|^2 +M_1 \|\epsilon_n\|^2 + M_2 = (1+M_1)\|\epsilon_n\|^2+M_2
\end{align*}
then we have 
\begin{align}\label{globalerrorinequality}
	\|\epsilon_{n+1}\|^2 \leq e^{NM_1} \Big(\frac{M_2}{M_1}+\|\epsilon_0\|^2\Big) -\frac{M_2}{M_1} = (e^{NM_1}-1) \frac{M_2}{M_1}.
\end{align}
The following theorem can assur the variable step size and the uniform convergence for solutions of the method.
\begin{theorem}\label{mainthm}
	Given that the singular boundary value problem in \eqref{Lane} satisfies the upper bounds assumption in $(17a)$-$(17d)$, then the successive approximation \eqref{numerical value y} with variable step sizes \eqref{stepsize} as $\hat{h}\to 0$, has $O((\hat{h})^2)$, converges uniformly in $n$ for $t_n < 1-\delta < 1$, and thus the global pointwise error for the above proposed algorithm is of order $O(\hat{h})$.
\end{theorem}
\begin{proof}
	If we have $N$ steps, \eqref{stepsize} gives $(1-\hat{h})^N=\delta$, and thus
	$N=\displaystyle\frac{\ln\delta}{\hat{h}}=\displaystyle\frac{-\ln\delta}{\ln(1-\hat{h})}$, whenever $h^*\to 0$.
	Then by using Lemma \ref{mainlemma}
	on \eqref{globalerrorinequality}, we have
	\begin{align*}
		\|\epsilon_{n}\|^2 
		&\leq\bigg[e^{\Big(-N \big(1+m_1(\hat{h})\big)\Big)}-1\bigg] \frac{m_2(\hat{h})^3}{1+m_1(\hat{h})} \\ &\leq \frac{D}{\delta^{m_1}}
		(\hat{h})^2
	\end{align*}
	where $D$ and $M_1$ are constants that do not depend on $n,\hat{h}$ or $\delta.$
\end{proof}



\section{\textbf{Simulation and Numerical Experiments}}
\maketitle
In this section we run the algorithm over some examples to show the validity of the method. We used MATLAB with bulit-in functions such as; {\em ode45} and \emph{EulerSolver}  

\begin{example}
	Consider the second order differential equation \eqref{Lane} with $a(t,x)=\sin x$, and\  $g(t,x)=x^5$, \ where the step size is $0.05$ and time interval $[0,1]$ along with initial conditions 
	$x(0)=0, \ x'(0)=2;$ i.e., 
	\begin{equation*}\frac{d^2 x}{d t^2}=\frac{-\sin x}{1-t}\frac{d x}{d t}+x^5 
	\end{equation*}
\end{example}
Table 1 compares the two dependent solutions $x(t)$ and $x'(t)$ for equation \eqref{Lane} given the above numerical values, and figures below draw the relationships between trajectories of the differential equation and the time.
\begin{table}[h]
	\begin{tabular}{|c|c|c|c|c|c|c|c|c|c|c|c|}
		\hline 
		Time & 0 & 
		0.1 & 
		0.2 & 
		0.3 & 
		0.4 & 
		0.5 & 
		0.6 & 
		0.7 & 
		0.8 & 
		0.9 & 
		1 \\\hline
		$x(t)$ & 0 & 
		0.2005 & 
		0.4055 & 
		0.62081 & 
		0.85335 & 
		1.1111 & 
		1.3984 & 
		1.7017 & 
		1.957 & 
		2.0312 & 
		1.7276 \\\hline
		$x'(t)$ & 2 & 
		2.0105 & 
		2.0682 & 
		2.186 & 
		2.3787 & 
		2.6495 & 
		2.9446 & 
		3.026 & 
		2.293 & 
		0.09581 & 
		-4.4824 \\\hline
	\end{tabular}
	\caption{the solutions $x, x'$ for Lane-Emden equation with time interval $[0,1].$}
\end{table}
The analytical solution to this problem is somewhat lower than our approximation. By shrinking the size of the interval $\Delta t$, we could calculate a more accurate estimate.
\begin{center}
	\includegraphics[width=.75\textwidth]{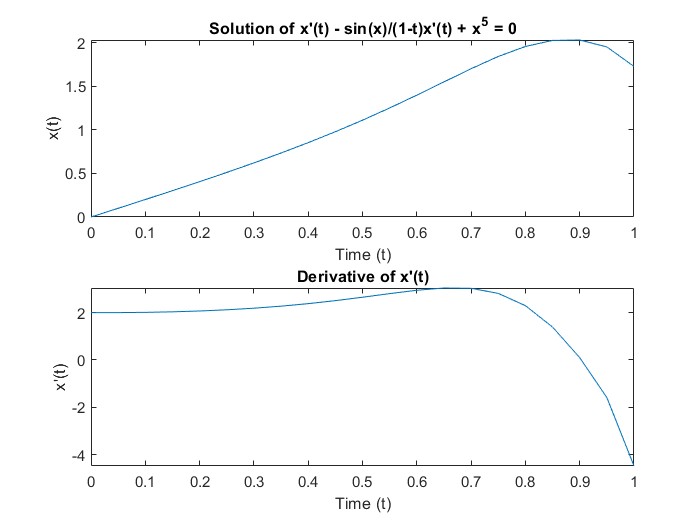}
\end{center}

\begin{example}
	Consider equation \eqref{Lane} with $a(t,x)=tx$, and\  $g(t,x)=x^3$, \ where the step size is $\Delta t=0.1$ and same time interval $[0,1]$ along with initial conditions
	$x(0)=0, \ x'(0)=2;$ i.e., 
	\begin{equation*}\frac{d^2 x}{d t^2}=\frac{tx}{1-t}\frac{d x}{d t}+x^3 
	\end{equation*}
	\begin{center}
		\includegraphics[width=.75\textwidth]{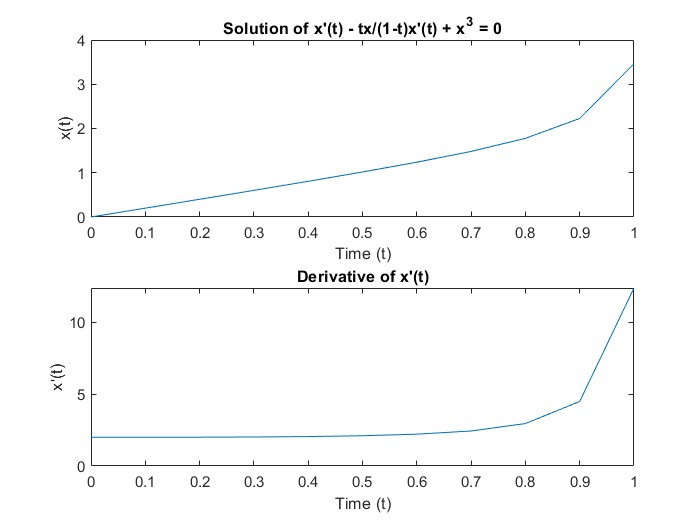}
	\end{center}
\end{example}

\begin{example}
	Consider a constant function $a(t,x)=3$ in Lane-Emden equation \eqref{Lane} with $g(t,x)=e^tx$, \ where the step size is $\Delta t=0.05$ and the time interval is $[0,2]$ along with initial conditions given as, $x(0)=0, \ x'(0)=2;$ i.e., 
	\begin{equation*}\frac{d^2 x}{d t^2}=\frac{3}{1-t}\frac{d x}{d t}-e^tx. 
	\end{equation*}
	\begin{center}
		\includegraphics[width=.75\textwidth]{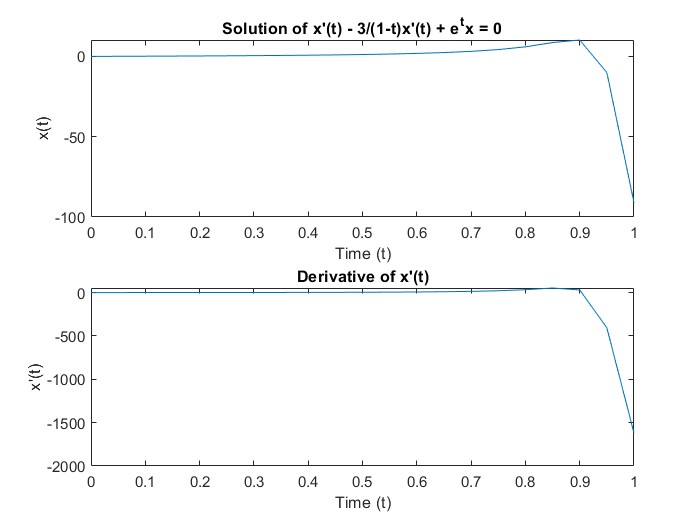}
	\end{center}
\end{example}

\begin{example}\label{example4}
	Consider the second-order dynamic equation \eqref{Lane} with $a(t,x)=2t$, and\  $g(t,x)=tx^2$, \ where the step size is $0.01$ (which can enlarged to help decrease the error estimates) and time interval $[0,1]$ along with initial conditions $x(0)=0, \ x'(0)=1;$ i.e., 
	\begin{equation*}\frac{d^2 x}{d t^2}=\frac{2t}{1-t}\frac{d x}{d t}+tx^2. 
	\end{equation*}
	\begin{center}
		\includegraphics[width=.75\textwidth]{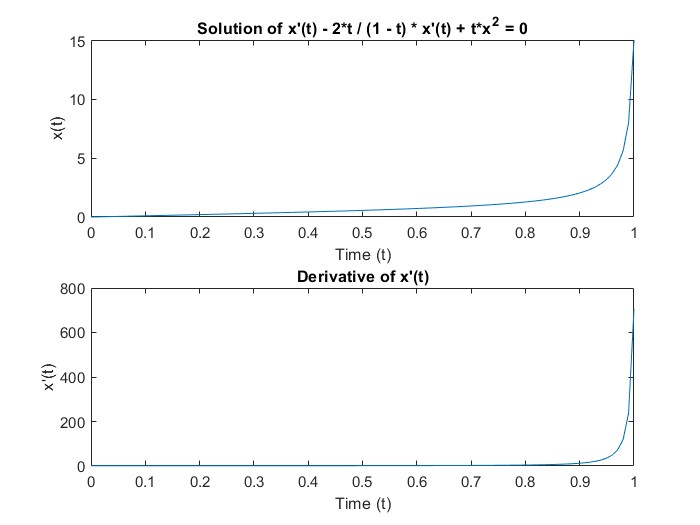}
	\end{center}
\end{example}
%
\vspace{.7cm}

\begin{figure}
	\begin{example}
		In this example we consider the non-autonomous inhomogeneous second order system with the right-hand side being $t^3e^{2t}, \ a(t,x)=4$, and\  $g(t,x)=4x$, \ where the step size is $0.01$ and along with initial conditions $x(0)=0, \ x'(0)=0;$ with the absence of singularity. The graphs shown below and the tables as well. 
	\end{example}
	\begin{minipage}{0.35\textwidth}
		\includegraphics[width=1.3\textwidth]{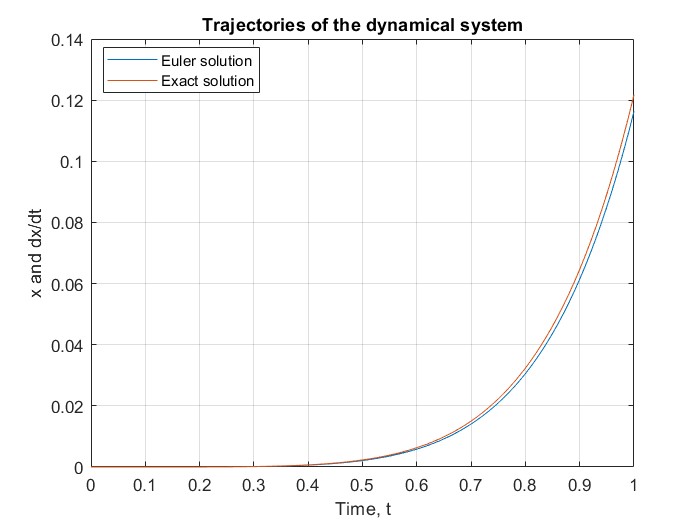}\caption{\small Comparison between approximated solution by Euler's method and the actual solution for the equation $x''+4x'+4x=t^3e^{2t}$.}
	\end{minipage}
	\hspace{1.1cm}
	\begin{minipage}{0.15\textwidth}
		\begin{tabular}{|l|c|c|c|c|}
			\hline
			t &	x & y(Euler) & y(exact) & Absolute error\\\hline
			0 &	0.500000 & 0.550000 & 0.588250 & 0.03825\\\hline
			0.2 & 0.618326 & 0.642485 & 0.662213 & 0.019728\\\hline
			0.4 & 0.678516 & 0.692098 & 0.703465 & 0.011367\\\hline
			0.3 & 0.712985 & 0.720934 & 0.727519 & 0.006585\\\hline
			0.4 & 0.732901 & 0.737205 & 0.740529 & 0.003324\\\hline
			0.6 & 0.742951 & 0.744533 & 0.745325 & 0.000792\\\hline
			0.8 & 0.745363 & 0.744600 & 0.743002 & 0.001598\\\hline 
			1.0 & 0.500000 & 0.425000 & 0.367225 & 0.057775\\\hline
			1.2 & 0.321304 & 0.283689 & 0.251965 & 0.031724\\\hline 
			1.4 & 0.224446 & 0.199932 & 0.177544 & 0.022388\\\hline 
			1.6 & 0.156632 & 0.136705 & 0.117386 & 0.019319\\\hline 
			1.8 & 0.098381 & 0.079456 & 0.060422 & 0.019034\\\hline 
			2.0 & 0.041116 & 0.021389 & 0.001065 & 0.020324\\\hline  
		\end{tabular}
	\end{minipage}
\end{figure}
\section{Conclusion and Extensions} In this paper our primary goal was to investigate the second-order singular Lane Emden type equations and we have successfully arrived at the solutions by the forward Euler's algorithm combined with the shooting method, which in turn, reduces the boundary value problem into initial value problem, so the method showed that it is a precise and time-saving method. The Lane Emden equations are solved for the values of the polytropic indices varies from 1, 2, 3 and 5 with having constants, linear functions and periodic functions in the drift term. The numerical solution of the problem for these values of indices replaces the unsolvable version of equation and any closed form solution that we wish to find. For the case of n = 2 the solution is obtained as an infinite power series. Graphical representations of these results give us information about polytropes for different values of polytropic indices which may be helful in the study of the behavior of stellar structures in astrophysics. One good extension for this work is through implementing backward Euler formula for a second-order differential equations where the recursion formula is the same, except that the dependent variable is a vector. Another possible modification for the work is by using the reliable Runge–Kutta method which promises accurate results in deriving the solutions of the Lane Emden equations. 
It is also significant in handling highly nonlinear differential equations with less computations and a larger  interval of convergence. For thinking globally, finite difference methods may be used to replace the shooting method to treat the boundary value problem. Finally, we may think of adding the additive noice to the second order differential equation (it will be called stochastic differential equation) and in this case, Euler's method will be replaced by Euler-Maruyama Algorithm, see, for instance, \cite{HughesGreg, HughesLoch}.


\subsection*{Acknowledgments and Declarations} The author would like to express his gratitude to professor Randy Hughes, a professor at Southern Illinois University Carbondale, for suggesting the problem and providing valuable advises along the way of writing this manuscript. The author declares that there was no conflict of interest or competing interest.
\medskip


%
%

\end{document}